\documentclass[11pt,a4paper]{amsart}

\usepackage[all]{xy}
\usepackage{amssymb}

 \theoremstyle{plain}

\theoremstyle{definition}
\newtheorem{theorem}{Theorem}[]

\newtheorem*{main}{Main Theorem}

\theoremstyle{definition}

\theoremstyle{definition}
\newtheorem{lem}[theorem]{Lemma}
\theoremstyle{definition}
\newtheorem{pro}[theorem]{Proposition}
\theoremstyle{definition}
\newtheorem{defn}[theorem]{Definition}
\theoremstyle{remark}

\newtheorem{que}{Question}

\def\F{\mathbb{F}}

\def\F{\mathbb{F}}

\keywords{p-solvable groups, p-length}

\subjclass{20D10 (20D15)}

\begin{document}

\title{A bound on the $p$-length of p-solvable groups}

\author{Jon Gonz\'alez-S\'anchez}
\address{Jon Gonz\'alez S\'anchez. Departamento de Matem\'aticas, Facultad de Ciencias, Universidad del Pa\'is Vasco--Euskal Herriko Unibertsitatea, Spain}
\email{jon.gonzalez@ehu.es}
\author{Francesca Spagnuolo}
\address{Francesca Spagnuolo. Departament dÕ\`Algebra, Universitat de Val\`encia; Dr. Moliner, 50; 46100, Burjassot, Val\`encia, Spain.}

\email{francesca.spagnuolo@uv.es}

\begin{abstract}
Let $G$ be a finite $p$-solvable group and $P$ a Sylow $p$-subgroup of $G$. Suppose that $\gamma_{\ell (p-1)}(P)\subseteq \gamma_r(P)^{p^s}$ for $\ell (p-1) < r+s(p-1)$, then the $p$-length is bounded by a function depending on $\ell$. 
\end{abstract}

\maketitle

\section{Introduction} 

All the group considered are finite. In the following $p$ will be a prime number.

A group is $p$-solvable if it has a sequence of  subgroups 
$$G=N_1\supset N_2\supset \ldots \supset N_k=1,$$
such that $N_{i+1}$ is a normal subgroup of $N_i$ and the index $|N_{i}:N_{i+1}|$ is either coprime to $p$ or a power of $p$. In such case the minimal number of factors $N_i/N_{i+1}$ which are $p$-groups is called  the $p$-length of $G$.
Alternatively a group $G$ is $p$-solvable if the upper series 
$$1\subseteq O_{p'}(G)\subseteq O_{p',p}(G)\subseteq O_{p',p,p'}(G)\subseteq \ldots$$
terminates at $G$ and we call the $p$-length of $G$ to the number of symbols $p$ appearing in the series. Recall that the previous series is constructed as follows: $O_{p'}(G)$ is the maximal normal $p'$-subgroup of $G$; $O_{p',p}(G)$ denotes the inverse image of the maximal normal $p$-subgroup of $G/O_{p'}(G)$; $O_{p',p,p'}(G)$ denotes the inverse image of the maximal normal $p'$-subgroup of $G/O_{p',p}(G)$, and so on.

Much is known about the $p$-length of $p$-solvable groups. For example, by a result of P. Hall and G. Higman explicit bounds of the $p$-length are known in terms of the derived length or the exponent of the Sylow $p$-subgroup (see \cite{HH}). More recently the first author and T. Weigel proved that if $p$ is odd and the elements of order $p$ of the Sylow subgroups are contained in the $p-2$-center of the Sylow p-subgroup, then the $p$-length is one (see \cite{JW}). This result was generalized by E. Khukhro in \cite{Ka} by proving that if the elements of order $p$ (or $4$ if $p=2$) of the Sylow $p$-subgroup are contained in the $\ell$-center of the Sylow $p$-subgroup then the $p$-length of the group is bounded in terms of $\ell$. In the same paper Khukhro also proved that if the Sylow $p$-subgroup is powerful, then the $p$-length is one, and he posted the question of whether this result can be generalized in the same way as the result for Sylow $p$-central subgroup. In this short note we will  try to answer this question by giving a bound of the $p$-length of a $p$-solvable group in terms of some power-commutator conditions in the Sylow $p$-subgroup.

\begin{main}
Let $G$ be a $p$-solvable group and $P$ a Sylow $p$-subgroup of $G$. Suppose that
$$\gamma_{\ell (p-1)}(P)\subseteq \gamma_r(P)^{p^s},$$
for $\ell (p-1) < r+s(p-1)$. Then the $p$-length of $G$ is bounded in terms of $l$. This condition holds in particular when $\gamma_{\ell (p-1)}(P) \subseteq P^{p^l}$.
\end{main}

The notation is standard in group theory. The subgroups $O_p(G)$ and $O_{p^\prime}(G)$ denote the maximal normal subgroup of order a power of $p$ or coprime to $p$ respectively. 
$[N,_kM]$ denotes the commutator subgroup $[N,M,\ldots ,M]$ where $M$ appears $k$ times.

Consider now a group $G$ and a normal elementary abelian $p$-group $V$. Let $\varphi$ be the action by conjugation of $G$ on $V$:
$$\varphi : G\times V\rightarrow V \ \ / \ \ (g, v)\rightarrow v^{g}$$
We can regard $V$ as a vector space over the field $\F _{p}$. In this case, the action by conjugation of $G$ on $V$ can be regarded as an action by linear transformation on this vector space, and we denote this action as follows:
$$(g, v)\in G\times V \rightarrow vT(g),$$
where $T(g)$ denotes the action of an element $g$ of $G$ on $V$.
If we denote by $1_{V}$ the identity transformation on $V$, then 
$$v(T(g)-1_{V}) = [v, g].$$



\section{Potent filtrations}

\label{sec:PF}

Let $P$ be a $p$-group, we say that a sequence of normal subgroups $(N_i)_{i=1}^k$
 is a \textit{Potent Filtration of type $\ell$} of $P$ if
\begin{enumerate}
\item $N_{i}\subseteq N_{j}$, for all $i>j$,
\item $N_k=1$,
\item $[N_{i}, P]\subseteq N_{i+1}$,  for  $i=1,\ldots ,k-1$,
\item $[N_{i},_l P]\subseteq N_{i+1}^{p}$, for $i=1,\ldots ,k-1$.
\end{enumerate}

Furthermore, we say that a group $N$ is \textit{PF--embedded of type $\ell$} in $P$ if there is a potent filtration of $P$ beginning at a subgroup $N$.

\begin{pro}
\label{filtration}Let $P$ be a pro-$p$ group and let $\{N_i\}_{i=1}^k$ 
a potent filtration of $P$ of type $\ell$.Then:
\begin{enumerate}
\item $[N_{i}^{p}, P]=[N_{i}, P]^{p}$ for all $i$
\item ${[N_i, P]}_{i=1,\ldots k}$ is a potent filtration of type $\ell$ of $P$
\item ${N_i^{p}}$ is a potent filtration of type $\ell$ of $P$
\end{enumerate}
\end{pro}

\begin{proof}
See \cite[Proposition 3.2]{O}.
\end{proof}

We continue with the case where $N$ is PF-embedded subgroup of type $p-2$.

\begin{pro}
Let $G$ be a $p$--solvable group and $P$ a Sylow $p$ subgroup of G. Suppose that $N$ is a PF-embedded subgroup of type $p-2$ of $P$. Then $N\subseteq O_{p'p}(G)$
\end{pro}

\begin{proof}
We can assume that $O_{p'}(G)=1$.  In order to simplify notations, we put $O_{p}(G)=V$. 
Modding out the Frattini subgroup  we can also assume that $V$ is an elementary abelian $p$-group. Furthermore, $V=C_{G}(V)$ \cite[Theorem 6.3.2]{Gor}. Therefore $G/V$ acts faithfully on $V$; so we can regard $V$ as a $\F _{p}(G/V)$--module. Moreover, we can consider the Jordan-Holder series of $V$ as an $\F_p[G/V]$-module:
$$V=V_{1}\supseteq V_{2}\supseteq V_{3}\supseteq \ldots \supseteq V_{n}$$ 
such that $V_{i}/V_{i+1}$ is simple for every $i$. Without loss of generality, we can take the quotient of the $\mathbb{F}_{p}[G/V]$--module $V$ over the second term of Jordan--Holder series and assume that $V$ is a simple $\mathbb{F}_{p}[G/V]$--module.

Take $(N_i)_{i=1}^k$ a potent filtration starting at $N$. We will prove by reverse induction on $i$ that for all $i$, $N_i\subseteq V$. For $i$ large enough it is clear. Suppose now that $N_{i+1}\subseteq V$. Take $v\in V$ and $n\in N_i$. Then
$$[v,\underbrace{n,\ldots ,n}_{p-2}] \in [V, \underbrace{N_i, \ldots N_i}_{p-2}] \subseteq [N_i,_{p-2}P]\subseteq N_i^{p}\subseteq V^p=1.$$
Therefore for all $v\in V$ and $n\in N_i$, we have
$$v(T(n)-1_{V})^{p-2}=0.$$
By \cite[Corollary 5.2]{JW} the size of Jordan blocks of $T(n)$ can only be 1.
Therefore $N_i\subseteq \ker(T)$ and by  \cite[Theorem 6.3.2]{Gor} $N_i\subseteq C_G(V)=V$.
\end{proof}

In order to prove the Main Theorem we will need a weaker result for PF-embedded subgroups of type $p-1$.

\begin{pro}
Let $G$ be a $p$--solvable group, $P$ a Sylow $p$ subgroup of $G$ and $N$ a PF-embedded subgroup of type $p-1$ of $P$. Then
\begin{enumerate}
\item[(a)] If $p\geq 5$, then $N^p\subseteq O_{p'p}(G)$.
\item[(b)] If $p=3$, then $N^{p^2}\subseteq O_{p'p}(G)$.
\item[(c)] If $p=2$, then $N\subseteq O_{p'p}(G)$.
\end{enumerate}
\end{pro}

\begin{proof}
As in the proof of Proposition 3, we can assume that $O_{p'}(G)=1$ and  $O_{p}(G)=V$ is a simple $\F _{p}(G/V)$--module. Note that $V=C_{G}(V)$ (see \cite[Theorem 6.3.2]{Gor}). Take $(N_i)_{i=1}^k$ a potent filtration starting at $N$.

(a) We will prove by reverse induction on $i$ that $N_i^p\subseteq V$. For $i$ large enough it is clear. Let us suppose that $N_{i+1}^p\subseteq V$. Take $n^p\in N_i^p$ and $v\in V$. Then by Proposition \ref{filtration}
$$[v, n^p, n^p]\in [V, N_i^{p}, N_i^{p}]\subseteq [P, N_i, N_i]^{p^{2}}\subseteq N_{i+1}^{p^2}\subseteq V^p=1.$$
Therefore 
$$v(T(n^p)-1_V)^2=1.$$
Then, as in Proposition 3 and since $p\geq 5$, the Jordan block size it can only can be $1$. So $N_i^{p}\subseteq  C_G(V)=V$.

(b) As in (a) we prove by reverse induction on $i$ that $N_i^{p^2}\subseteq V$. For $i$ large enough it is clear. Let us suppose that $N_{i+1}^{p^2}\subseteq V$. Take $n^{p^2}\in N_i^{p^2}$ and $v\in V$.
Then, by \cite[Theorem 2.4]{O},
$$[v, n^{p^{2}}]\in [V, N_i^{p^{2}}]\subseteq  [V,_{p^{2}} N_i]\subseteq N_{i+1}^{p^{3}}\subseteq V^p=1.$$
Therefore $N_i^{p^2}\subseteq C_G(V)=V$.

(c)  As in (a) we prove by reverse induction on $i$ that $N_i\subseteq V$. For $i$ large enough, it is clear. Let us suppose that $N_{i+1}\subseteq V$. Take $n\in N_i$ and $v\in V$.
Then

$$[n,v]\in [N_i,V]\subseteq N_{i+1}^2\subseteq V^p=1.$$

Therefore $n\in C_G(V)=V$ and $N_i\subseteq V$.
\end{proof}

\section{Proof of the main theorem}

Now we introduce a family ${E_{k,r}}(P)$ of subgroups of a finite $p$-group $P$ as in  \cite{O}.
\begin{defn}
Let $P$ be a finite $p$-group. For any pair $k$, $r$ of positive integers, we define the subgroup
$$E_{k,r}(P)= \prod_{i+j(p-1)\geq k\ \text{and}\ i\geq r} \gamma _{i}(P)^{p^{j}}$$
\end{defn}

The next theorem is a stronger version of the Main Theorem.

\begin{theorem}
\label{main}
Let $G$ be a $p$-solvable group and let $P$ be a Sylow $p$-subgroup of $G$. Suppose that 
$\gamma _{\ell (p-1)}(P)\subseteq E_{\ell (p-1)+1,1}(P)$, then the $p$--length of $G$ is bounded in terms of $\ell$.
\end{theorem}

\begin{proof}
Put $E=E_{(\ell-1)(p-1),1}(P)$. By \cite[Theorem 4.8]{O} the subgroup $E$ is PF--embedded of type $p-1$ in P. So $E^{p^2}\subseteq O_{p',p}(G)$.

By construction the exponent of $P/O_{p',p}(G)$ is at most the exponent of $P/E^{p^2}$ which is bounded by $p^{\ell+1}$. Therefore by Theorem A of \cite{HH} the $p$-length of $G$ is bounded in terms of $\ell$.
\end{proof}

\section{A question and a lemma}

In Section \ref{sec:PF} we have proved that if $G$ is a $2$-solvable group, $P$ is a Sylow $2$-subgroup of $P$ and $N$ is a PF-embedded subgroup of $P$, then $N\subseteq O_{p,p^\prime} (G)$. For the case were $p$ is an odd prime we have only a weaker version of the result.  Therefore the following question arises naturally.

\begin{que}
\label{Q1}
Let $G$ be a $p$-solvable group, $P$ is a Sylow $p$-subgroup of $P$ and $N$ is a PF-embedded subgroup of $P$. Then $N\subseteq O_{p,p^\prime} (G)$.
\end{que}

A positive answer to this question will provide an improvement of the implicit bound in Theorem \ref{main}. The following lemma could be helpful. It is a generalization of  \cite[Theorem 6.3.2]{Gor}.

\begin{lem}
Let $G$ be a $p$-solvable group such that $O_{p'}(G)=1$. Let $N$ be a normal subgroup of $G$ such that there exists an integer $l$  for which $[O_p(G),_{l}N]=1$. Then $N\subseteq O_p(G)$.
\end{lem}

\begin{proof}
Put $V=O_p(G)$ and denote by $p^r$ the exponent of $V$. Notice that $O_{p^\prime}(N)=1$ and put $M=O_{p^\prime ,p}(N)$. 
Then by \cite[Theorem 2.4]{O}
$$[V, M^{p^{r+l}}]\subseteq [V,M]^{p^{r+l}} \prod_{i=1}^{r+l} [V,_{p^i}M]^{p^{r+l-i}}.$$
In the previous equation either $p^i\geq l$ or $r+l-i\geq r$. Therefore $[V, M^{p^{r+l}}]=1$.

Let $H$ be a nontrivial $p^\prime$-Hall subgroup of $M$. Then
$$[V, H]=[V, H^{p^{r+l}}]\subseteq [V, M^{p^{r+l}}]=1.$$
Therefore $H$ centralizes $V$ and $M=H\times V$. In particular $H\subseteq O_{p'}(G)=1$ and therefore $O_{p^\prime ,p}(N)=O_p(N)$ which implies that $N\subseteq O_p(G)$.
\end{proof}

\end{document}